\documentclass[11pt]{amsart}

\usepackage{amssymb}

\usepackage{enumerate}

\makeatletter
\@namedef{subjclassname@2010}{%
  \textup{2010} Mathematics Subject Classification}
\makeatother

\newtheorem{thm}{Theorem}[section]

\newtheorem{lemma}[thm]{Lemma}

\theoremstyle{definition}

\numberwithin{equation}{section}

\usepackage[right=3cm,left=3cm]{geometry}

\newcommand{\R}{\mathbb{R}}
\newcommand{\C}{\mathbb{C}}

\newcommand{\Z}{\mathbb{Z}}

\newcommand{\abs}[1]{\mathord{\left|#1\right|}}
\newcommand{\nrm}[1]{\mathord{\left\lVert #1 \right\rVert}}
\newcommand{\eps}{\varepsilon}
\newcommand{\zsum}{\sideset{}{^*}\sum}
\newcommand{\eand}{\,\,\,\text{ and }\,\,\,}
\newcommand{\floor}[1]{\left\lfloor #1 \right\rfloor}
\newcommand{\nsum}{\sideset{}{'}\sum}

\DeclareMathOperator{\Eit}{\tilde Ei}
\DeclareMathOperator{\li}{li}
\DeclareMathOperator{\Ei}{Ei}

\begin{document}

\baselineskip=17pt

\title{On the first sign change in Mertens' Theorem}
\author{Jan B\"uthe}
\address{Hausdorff Center for Mathematics, Endenicher Allee 62, 53115 Bonn}
\email{jan.buethe@hcm.uni-bonn.de}
\date{\today}

\begin{abstract}
 The function
$\sum_{p\leq x} \frac{1}{p} - \log\log(x) - M$ is known to change sign infinitely often, but so far all calculated values are positive. In this paper we prove that the first sign change occurs well before $\exp(495.702833165)$.
\end{abstract}
\maketitle

\section{Introduction}

Mertens' Theorem states that
\begin{equation*}
\Delta_M(x) := \sum_{p\leq x} \frac{1}{p} - \log\log(x) - M = O(\log(x)^{-1})
\end{equation*}
 for $x\to\infty$, where $M= 0.26149\dots$ denotes the Mertens constant \cite{Mertens74}. Rosser and Schoenfeld observed that $\Delta_M(x)$ is always positive for $1\leq x\leq 10^8$ and posed the question whether this would always be the case \cite[p. 72f]{RS62}. This has been answered by Robin who showed that $\Delta_M(x)$ changes sign infinitely often \cite{Robin83}.

In this paper we show that the first sign change occurs before $\exp(495.702833165) = 1.909875\dots \times 10^{215}$. More specifically, we prove
 \begin{thm}\label{t:explicit-upper-bound}
 There exists an $x_0\in [\exp(495.702833109),\exp(495.702833165)]$ such that $\Delta_M(x)$ is negative for all $x\in [x_0-\exp(239.046541),x_0]$.
 \end{thm}
 
This problem is similar to bounding the Skewes number, the number in $[2,\infty)$ where the first sign change of $\Delta(x) = \pi(x)-\li(x)$ occurs \cite{Skewes1933}; this number is by now known to lie between $10^{19}$ \cite{BuetheDiss} and $\exp(727.951335792)$ \cite{STD15}. The functions $\Delta(x)$ and $\Delta_M(x)$ are closely related and the Prime Number Theorem, $\Delta(x) = o(\li(x))$ for $x\to\infty$, is in fact equivalent to $\Delta_M(x) = o(\log(x)^{-1})$ for $x\to\infty$. But since $\Delta(x)$ and $\Delta_M(x)$ are biased in opposite directions there is no correlation between the sign changes of the two functions. On the Riemann Hypothesis, sign changes of $\Delta_M(x)$ rather occur at points where $\Delta(x) \approx -2\sqrt{x}/\log(x)$.
 
Theorem \ref{t:explicit-upper-bound} is proven by an adaption of the Lehman method for bounding the Skewes number \cite{Lehman66}, using explicit formulas and numerical approximations to part of the zeros of the Riemann Zeta Function from \cite{FKBJ}. In doing so, the kernel function in Lehman's method is replaced by the Logan function \cite{Logan88}, which appears to be more suitable for this problem. This is done in such generality that it can easily be reapplied to the original Lehman method.

\section{Notations}
As usual $\zeta(s)$ denotes the Riemann zeta function and zeros of $\zeta(s)$ are denoted by $\rho = \beta + i\gamma$ with $\beta,\gamma\in\R$. The Euler constant is denoted by $C_0=0.57721\dots$ and the Mertens constant by
\begin{equation}\label{e:Mertens-constant}
M=C_0 -  \sum_{p}\sum_{m=2}^\infty \frac{1}{m\,p^m}= 0.26149\dots.
\end{equation}
We use the symbol $\nsum$ to define normalized summatory functions, i.e. we define
\begin{equation*}
\nsum_{x< n < y} a_n := \frac 12 \sum_{x< n<y} a_n + \frac 12 \sum_{x\leq n\leq y} a_n.
\end{equation*}
Moreover, we define the Mertens prime-counting functions
\begin{equation*}
\pi_M(x) =\nsum_{p<x} \frac 1p \eand \pi_M^*(x) = \sum_{m=1}^\infty \frac{\pi_M(x^{1/m})}{m}.
\end{equation*}
The Fourier transform of a function $f$ is denoted by $\hat f$ and defined by
\begin{equation*}
\hat f(x) = \int_{-\infty}^\infty f(t) e^{-itx}\, dt.
\end{equation*}
Finally, we will use Turing's big theta notation for explicit estimates and write $f(x) = \Theta(g(x))$ for $\abs{f(x)} \leq g(x)$.

\section{Description of the Method}
The method we use is similar to the Lehman method for finding regions where $\pi(x)-\li(x)$ is positive \cite{Lehman66}. We aim to calculate upper bounds for a weighted mean value
\begin{equation}\label{e:weighted-mean}
\int_{\omega-\eps}^{\omega+\eps} K(y-\omega) y e^{y/2} \bigl[\pi_M(e^y) - \log(y) - M\bigr]\, dy ,
\end{equation}
where $K(y)$ is a non-negative kernel function. By using explicit formulas this mean value can be expressed as a sum over the non-trivial zeros of $\zeta(s)$, which can be approximated numerically. Then, if an $\omega$ can be found for which the value in \eqref{e:weighted-mean} is negative, there must exist an $x \in [\exp(\omega - \eps),\exp(\omega +\eps)]$ such that $\pi_M(x) - \log\log(x) - M$ is negative.

Lehman's method uses the Gaussian function as a kernel function but we prefer to use dilatations of the function
\begin{equation*}
K_{c} (y) := 
\begin{cases}
\frac{c}{2\sinh(c)}I_0(c\sqrt{1-y^2}) &\abs{y}<1, \\
0 &\text{otherwise,}
\end{cases}
\end{equation*}
where $I_0(t) := \sum_{n=0}^\infty (t/2)^{2n}/(n!)^2$ denotes the $0$-th modified Bessel function. The Fourier transform of $K_c$ is given by the Logan function (see \cite[Proposition 4.1]{FKBJ})
\[
\hat K_{c}(x) = \ell_{c}(x) := \frac{c}{\sinh c} \frac{\sin(\sqrt{t^2-c^2})}{\sqrt{t^2-c^2}},
\]
which satisfies an optimality property well-suited for this problem \cite{Logan88}, and which outperforms the Gaussian function in the similar context of calculating the prime-counting function analytically \cite{FKBJ}.

We define
\begin{equation*}
K_{c,\eps}(y) := \frac{1}{\eps}K_c(y/\eps) \quad\text{and}\quad \ell_{c,\eps}(x) := \hat K_{c,\eps}(x) = \ell_{c}(\eps x).
\end{equation*}
Then our main result is

\begin{thm}\label{t:main-theorem}
Let $0< \eps< 10^{-3}$, $c\geq 3$, $\omega - \eps > 200$, and let $H\geq c/\eps$ be a number such that $\beta = 1/2$ holds for all zeros $\rho=\beta + i\gamma$ of the Riemann zeta function with $0<\gamma \leq H$. Furthermore, let $h=0$ if the Riemann hypothesis holds and $h=1$ otherwise. Then we have
\begin{multline}\label{e:mean-upper-bound2}
\int_{\omega-\eps}^{\omega+\eps} K_{c,\eps}(y-\omega) \, y \, e^{y/2} \bigl[\pi_M(e^y) - \log(y) - M\bigr]\, dy \\ \leq
\sum_{\abs{\gamma} \leq c/\eps} e^{-i\gamma \omega} \ell_{c,\eps}(\gamma)\Bigl(\frac{1}{\rho} - \frac{1}{\omega\rho^2}\Bigr) \\
 +1 + 5.4\times 10^{-10} + \mathcal E_1 + \mathcal E_2 + \mathcal E_3,
\end{multline}
where
\begin{align}
\mathcal E_1 &\leq  0.33 \, e^{h \omega/2} \, \frac{e^{0.71\sqrt{c\eps}}}{\sinh c} \log(3c) \log\Bigl(\frac{c}{\eps}\Bigr) , \label{e:E1-bound}\\
\mathcal E_2 &\leq \frac{3.36+126\,\eps}{1000\,\omega^2} + 2.8 \Bigl(\frac{e}{2H}\Bigr)^{\omega/2-1} \log(H) ,\label{e:E2-bound} \\
\intertext{and}
\mathcal E_3 &\leq \frac{e^{\omega/2}}{1.99 H}\log(H) \left( \frac{c\, e^{3.12\sqrt{c\eps}}}{\omega \sinh(c)} +  \Bigl(\frac{e\eps}{\omega}\Bigr)^{\omega/2} \right).\label{e:E3-bound}
\end{align}
Moreover, if $a\in(0,1)$ satisfies $ac/\eps \geq 10^3$ in addition to the previous conditions, then
\begin{equation}\label{e:ace}
\sum_{\frac{ac}{\eps} <\abs{\gamma} \leq  \frac c\eps}	\abs{ e^{-i\gamma \omega} \ell_{c,\eps}(\gamma)\Bigl(\frac{1}{\rho} - \frac{1}{\omega\rho^2}\Bigr)} \leq \frac{0.32 + 3.51 c\eps}{c a^2} \log\Bigl(\frac c\eps\Bigr) \frac{\cosh(c\sqrt{1-a^2})}{\sinh(c)}.
\end{equation}
\end{thm}

The proof needs some preparation.

\section{The explicit formula for $\pi^*_M(x)$}
The first ingredient is the explicit formula for $\pi^*_M(x)$. We define the auxiliary function
\begin{equation*}
\Eit(z) = \int_{0}^\infty \frac{e^{z-t}}{z-t}\, dt,
\end{equation*}
which coincides with the exponential integral $\Ei(z)$ in $\R\setminus\{0\}$, and which occurs naturally in explicit formulas for prime-counting functions.

\begin{lemma}\label{l:explicit-formula}
Let $x>1$. Then
\begin{equation}\label{e:explicit-formula}
\pi_M^*(x) = \log\log(x) + C_0 - \zsum_\rho \Eit(-\rho\log x) + \int_{x}^\infty \frac{dt}{t^2\log(t)(t^2-1)},
\end{equation}
where the star indicates that the sum over zeros is calculated as
\begin{equation*}
\lim_{T\to\infty} \sum_{\abs{\gamma}<T} \Eit(-\rho\log x).
\end{equation*}
\end{lemma}
\begin{proof}
The argument is similar to the original proof of the Riemann explicit formula in \cite{Mangoldt95}. Let
\begin{equation}\label{e:Cheb-psi}
\psi(x,r) = \nsum_{p^m<x} \frac{\log p}{p^{mr}}.
\end{equation}
Then we have
\[
\pi^*_M(x) = \int_{1}^\infty \psi(x,r)\, dr.
\]
From \cite[(39)]{Landau1908} we get the explicit formula
\[
\psi(x,r) = \frac{x^{1-r}}{1-r} - \zsum_{\rho}\frac{x^{\rho-r}}{\rho-r} - \sum_{n=1}^\infty \frac{x^{-2n-r}}{-2n-r} - \frac{\zeta'}{\zeta}(r).
\]
Since $\Ei(-x) = \log(x) + C_0 + o(x)$ for $x\searrow 0$ \cite[p. 40]{Olver97}, and since $\log \zeta(1+\eps) = -\log(\eps) +o(1)$ for $\eps\searrow 0$ we have  
\[
\int_{1}^\infty \frac{x^{1-r}}{1-r} - \frac{\zeta'}{\zeta}(r)\, dr = \lim_{\eps\searrow 0} \Bigl[\Ei(-\eps\log x) + \log\zeta(1+\eps)\Bigr] = \log\log(x) + C_0.
\]
The sum over zeros takes the form
\[
\int_{1}^\infty \zsum_\rho \frac{x^{\rho-r}}{\rho-r}\, dr = \zsum_\rho \Eit((\rho-1)\log x) = \zsum_\rho \Eit(-\rho\log x),
\]
and for the sum over the trivial zeros we find
\[
\int_1^\infty \sum_{n=1}^\infty \frac{x^{-2n-r}}{2n+r}\, dr = \int_{1}^\infty \sum_{n=1}^\infty x^{-(2n+1) r} \frac{dr}{r} = \int_1^\infty \frac{x^{-3r}}{1-x^{-2r}}\, dr = \int_x^\infty \frac{dt}{t^2\log(t)(t^2-1)}.\qedhere
\]
\end{proof}

\section{The difference $\pi^*_M(x) - \pi_M(x)$}
By definition of the Mertens constant \eqref{e:Mertens-constant} we have
\begin{equation*}
\pi_M(x) = \pi^*_M(x) +M - C_0 + r_M(x),
\end{equation*}
where
\begin{equation*}
r_M(x) = \nsum_{\substack{p^m > x\\m\geq 2}} \frac{1}{m p^m}.
\end{equation*}
The term $r_M(x)$ is responsible for the positive bias in Mertens' Theorem and needs to be bounded from above.

\begin{lemma}\label{l:difference}
Let $\log(x)>200$. Then
\begin{equation*}
 r_M(x) \leq \frac{1 + 5.3\times 10^{-10}}{\sqrt{x}\log x}.
\end{equation*}
\end{lemma}
\begin{proof}
First we consider the contribution of the squares of prime numbers which yield the main term. Let $r(t) =\psi(t)-t$, where $\psi(t) := \psi(t,0)$ in the sense of \eqref{e:Cheb-psi} denotes the normalized Chebyshov function, and assume $\abs{r(t)} < \eps t$ for $t\geq \sqrt{x}$ and some $\eps>0$. Then partial summation gives
\begin{equation}\label{e:psquare}
\nsum_{p>\sqrt{x}} \frac{1}{p^2} < \left[\frac{-r(t)}{t^2\log t}\right]_{\sqrt{x}}^\infty + \int_{\sqrt x}^\infty \frac{dt}{t^2\log(t)} - \int_{\sqrt{x}}^\infty r(t) \frac{d}{dt}\Bigl( \frac{1}{t^2\log t}\Bigr)\, dt < 2 \frac{1+3\eps}{\sqrt{x}\log x}.
\end{equation}

For $3\leq m \leq \log(x)$ we use
\begin{equation*}
\sum_{p\geq x^{1/m}} \frac{1}{p^m} \leq \frac{1}{x} + \int_{x^{1/m}}^\infty \frac{dt}{t^m} = \frac{1}{x} + \frac{1}{m-1}x^{1/m-1},
\end{equation*}
which gives
\begin{equation*}
\sum_{\substack{p^m\geq x \\3\leq m\leq \log x}} \frac{1}{m p^m} \leq \frac{\log x}{x} + (\zeta(2)-1) x^{-2/3} < \frac{10^{-12}}{\sqrt{x}\log(x)}.
\end{equation*}

For $m>\log x$ we estimate trivially:
\begin{equation*}
\sum_{p} \frac{1}{p^m} \leq \sum_{n=3}^\infty n^{-m} + 2^{-m} \leq 2^{-m} + \int_{2}^\infty \frac{dt}{t^m} = 2^{-m}\Bigl(1 + \frac{2}{m-1}\Bigr).
\end{equation*}
Therefore, we get
\begin{equation*}
 \sum_{\substack{p^m\geq x \\ m > \log x}} \frac{1}{m p^m}\leq \frac{1.01}{\log(x)} \sum_{m\geq \log x} 2^{-m} \leq \frac{2.02\times 2^{-\log(x)}}{\log(x)} < \frac{10^{-16}}{\sqrt{x}\log(x)}.
\end{equation*}

By \cite[Table 1]{BuetheB} \eqref{e:psquare} holds with $\eps = 1.752\times 10^{-10}$ and so the assertion follows.
\end{proof}

\section{Evaluating the sum over zeros}
The next problem is to approximate the following integral of the sum over zeros
\begin{equation*}
\int_{-\eps}^\eps K_{c,\eps}(y-\omega) \, y\, e^{y/2} \zsum_\rho \Eit(-\rho y)\, dy.
\end{equation*}
Here, integral and sum may be interchanged, since the sum converges locally in $L^1$. Therefore, we may treat each summand individually.

\subsection{Asymptotic Expansion of the Summands}
Since the Logan kernel should also be of interest for the question on finding regions where $\pi(x)-\li(x)$ is positive, the following Lemma is presented in a more general version, which also covers the classical case.

\begin{lemma}\label{l:I-asymp}
Let $0< \eps < \omega$, and let $K\in L^1([-\eps,\eps])$ satisfy $\nrm{K}_{L^1}=1$. Let $a\in[0,1]$, let $\rho = \beta + i \gamma$, where $0\leq \beta\leq 1$ and $\gamma\in\R\setminus\{0\}$, and let
\begin{equation*}
\Phi_{\omega,\rho,a} = \int_{\omega-\eps}^{\omega+\eps} K(y-\omega) \, y \, e^{(\frac 12 - a)y} \Eit((a-\rho) y)\, dy.
\end{equation*}
Then we have
\begin{equation}\label{e:asymp1}
\Phi_{\omega,\rho,a} =  \sum_{j=1}^k (j-1)!\frac{F_{\omega,\rho}^{(-j)}(0)}{(\rho-a)^j} + \Theta\left(\frac{k! \, e^{\frac \eps 2} \, e^{(\frac 12 - \beta)\omega}}{(\omega-\eps)^k \, \abs{\gamma}^{k+1}}\right),
\end{equation}
where	$F_{\omega,\rho}^{(-1)}(0) = - e^{(\frac 12 - \rho)\, \omega} \hat K(\frac{\rho}{i}-\frac{1}{2i})$ and for $j\geq 2$ and any $m\geq 0$,
\begin{multline}\label{e:F-expansion}
F_{\omega,\rho}^{(-j)}(0) = (-1)^{j} e^{(\frac 12-\rho)\,\omega} \sum_{n=0}^m  \binom{n+j-2}{n} \frac{(-i)^n\hat K^{(n)} (\frac{\rho}{i}-\frac{1}{2i})}{\omega^{n+j-1}} \\
+ \Theta\left(\frac{e^{j-2+\frac \eps 2} e^{(\frac 12 -\beta) \omega}}{\omega^{j-1}}\frac{(e\eps/\omega)^{m+1}}{1-e\eps /\omega}\right).
\end{multline}
\end{lemma}

\begin{proof}
By definition of $\Eit$ we have
\begin{align}
\Phi_{\omega,\rho,a}	&= \int_{\omega-\eps}^{\omega+\eps} K(y-\omega) \, y\, e^{(\frac 12 -a) y} \int_0^\infty \frac{e^{(a-\rho-r)y}}{a-\rho-r} \, dr\, dy \notag\\
	&= \int_0^\infty \frac{1}{a-\rho - r}  \int_{\omega-\eps}^{\omega+\eps} K(y-\omega) \, y\, e^{(\frac 12 -\rho -r)y} \, dy\, dr.
\end{align}

Now let 
\begin{equation*}
F^{(-j)}_{\omega,\rho}(r) := (-1)^j\int_{\omega-\eps}^{\omega+\eps} y^{1-j} K(y-\omega) e^{(\frac 12 -\rho -r)y}\, dy,
\end{equation*}
which is well defined since $\omega> \eps$, and satisfies $\frac d {d r} F_{\omega,\rho}^{(-j)} = F_{\omega,\rho}^{(1-j)}$. Then partial summation gives
\begin{equation*}
\Phi_{\omega,\rho,a} = -\int_0^\infty \frac{F_{\omega,\rho}^{(0)}(r)}{r + \rho -a}\, dr = \sum_{j=1}^k (j-1)! \frac{F_{\omega,\rho}^{(-j)}(0)}{(\rho-a)^j} - k! \int_0^\infty \frac{F_{\omega,\rho}^{(-k)}(r)}{(r+\rho-a)^{k+1}}\, dr. 
\end{equation*}

Here, the trivial bound
\begin{equation*}
\abs{F_{\omega,\rho}^{(-k)}(r) } \leq \int_{-\eps}^{\eps} \frac{\abs{K(y)}}{(\omega+y)^{k-1}} e^{(\frac 12 - \beta -r)(y+\omega)}\, dy \leq \frac{e^{\frac\eps 2}}{(\omega-\eps)^{k-1}} e^{(\frac 12 -\beta)\omega} e^{r(\eps-\omega)}
\end{equation*}
yields
\begin{equation*}
\int_0^\infty \abs{\frac{F_{\omega,\rho}^{(-k)}(r)}{(r+\rho-a)^{k+1}}}\, dr \leq \frac{e^{\frac\eps 2} e^{(\frac 12 -\beta)\omega}}{(\omega-\eps)^{k}\,\abs{\gamma}^{k+1}}
\end{equation*}
which confirmes \eqref{e:asymp1}. It remains to evaluate $F_{\omega,\rho}^{(-j)}(0)$. For $j=1$ we find
\begin{equation*}
F_{\omega,\rho}^{(-1)}(0) = -e^{(\frac 12 - \rho)\omega} \int_{-\eps}^\eps K(y) e^{-i (\frac \rho i - \frac 1 {2i})y}\, dy = - e^{(\frac 12 - \rho)\omega} \hat K\Bigl(\frac \rho i - \frac 1 {2i}\Bigr).
\end{equation*}
For larger values of $j$ we use the Taylor series expansion
\begin{equation*}
\frac{1}{(\omega+y)^u} = \sum_{n=0}^\infty \binom{u+n-1}{n} \frac{(-y)^n}{\omega^{u+n}}
\end{equation*}
and
\begin{equation}\label{e:khat-der}
\int_{-\eps}^\eps K(y) y^n e^{-i(\frac \rho i - \frac 1{2i})y} \, dy = i^n \hat K^{(n)}\Bigl(\frac \rho i - \frac 1 {2i}\Bigr),
\end{equation}
which gives
\begin{equation*}
F_{\omega,\rho}^{(-j)}(0) = (-1)^j e^{(\frac 12 -\rho)\omega} \sum_{n=0}^\infty \binom{j+n-2}{n} \frac{(-i)^n\hat K^{(n)}(\frac \rho i  - \frac 1{2i})}{\omega^{n+j-1}}.
\end{equation*}
From \eqref{e:khat-der} we get
\begin{equation*}
\abs{\hat K^{(n)}\Bigl(\frac \rho i  - \frac 1{2i}\Bigr) }\leq e^{\frac{\eps}{2}} \eps^{n}
\end{equation*}
and the inequality $\binom{a}{b} \leq (\frac{e a}{b})^b$, which follows from Stirling's lower bound for $b!$, implies
\begin{equation*}
\binom{j+n-2}{n} \leq e^{n} \left(1 + \frac{j-2}{n}\right)^n \leq e^{n+j-2}.
\end{equation*}
Thus, we have
\begin{equation*}
\sum_{n=m+1}^\infty  \binom{j+n-2}{n} \frac{\abs{\hat K^{(n)}(\frac \rho i  - \frac 1{2i}) }}{\omega^{n+j-1}} \leq \frac{e^{j-2 + \frac \eps 2}}{\omega^{j-1}} \sum_{n=m+1}^\infty \Bigl(\frac{e\eps}{\omega}\Bigr)^n = \frac{e^{j-2 + \frac \eps 2}}{\omega^{j-1}} \frac{(e\eps/\omega)^{m+1}}{1- e\eps/\omega},
\end{equation*}
which confirms the bound in \eqref{e:F-expansion}.
\end{proof}

\subsection{Bounds for the Kernel Function}
We need some bounds to estimate the tails of the sum over zeros. These are provided by the following two Lemmas from \cite{BuetheA} and \cite{BuetheB}:
\begin{lemma}[{\cite[Lemma 2]{BuetheB}}]\label{l:logan-tail-bounds}
Let $0<\eps < 10^{-3}$ and $c\geq	 3$. Then we have
\begin{equation}\label{e:logan-sum1}
\sum_{\abs{\gamma} >\frac c\eps} \frac{\abs{\ell_{c,\eps}\bigl(\frac{\rho}{i}- \frac{1}{2i}\bigr)}}{\abs{\gamma}} \leq 0.32 \frac{e^{0.71 \sqrt{c\eps}}}{\sinh(c)} \log(3c) \log\Bigl(\frac c\eps\Bigr).
\end{equation}
\end{lemma}
\begin{lemma}[{\cite[Lemma 4.5]{BuetheA}}]
Let $0<\eps < 10^{-3}$ and $c\geq	 3$, and let $a\in(0,1)$ satisfy $ac/\eps>10^3$. Then we have
\begin{equation}\label{e:logan-sum2}
\sum_{\frac{ac}{\eps}<\abs{\gamma}\leq \frac c\eps} \frac{\abs{\ell_{c,\eps}(\gamma)}}{\abs{\gamma}} \leq \frac{1+11c\eps}{\pi c a^2}\log\Bigl(\frac c\eps\Bigr) \frac{\cosh(c\sqrt{1-a^2})}{\sinh(c)}.
\end{equation}
\end{lemma}

We also need bounds for the derivatives $\ell^{(n)}_{c,\eps}(\frac{\rho}{i}-\frac{1}{2i})$ occurring in \eqref{e:F-expansion}, for calculations not assuming the Riemann hypothesis.

\begin{lemma}\label{l:logan-bound}
Let $0<\eps \leq \delta < c/100$, and let $z\in \C$ satisfy $\abs{\Re(z)}\geq c/\eps$ and $\abs{\Im(z)} \leq \frac 12$. Then
\begin{equation*}
\abs{\ell^{(n)}_{c,\eps}(z)} \leq n! \frac{c\, e^{1.56\sqrt{\delta c}}}{\sinh(c)} \Bigl(\frac{2\eps}{\delta}\Bigr)^n.
\end{equation*}
\end{lemma}

\begin{proof}
The bound follows from the Cauchy formula
\begin{equation*}
\ell^{(n)}_{c,\eps}(z) = \frac{n!}{2\pi i} \oint_{\abs{z-\xi}=\frac{\delta}{2\eps}} \frac{\ell_{c,\eps}(\xi)}{(z-\xi)^{n+1}}\, d\xi
\end{equation*}
if we show that
\begin{equation}\label{e:logan-bound}
\abs{\ell_{c,\eps}(\xi)} \leq \frac{c\, e^{1.56\sqrt{\delta c}}}{\sinh(c)}
\end{equation}
in the range of integration. By basic properties of $\ell_{c,\eps}$ it suffices to prove this bound for $\eps=1$ under the conditions $\Re(\xi) \geq c-\delta,$ $0\leq \Im(\xi) \leq \delta$, and we may also assume $\delta < c/100$. Since we have
\begin{align*}
\abs{\Im(\sqrt{\xi^2-c^2})} &\leq \abs{\Im(\sqrt{(c-\delta + i \delta)^2-c^2})} \\
 &\leq \sqrt{2\abs{1+i}\delta c} \, \sin\Bigl(\frac \pi 4 + \frac{1}{2} \arctan\Bigl(\frac{\delta c-\delta^2}{\delta c}\Bigr)\Bigr) \\
 &\leq 2^{3/4} \sin(1.181) \sqrt{\delta c} \leq 1.56 \sqrt{\delta c}
\end{align*}
under these conditions, the desired bound follows from
\begin{equation*}
\abs{\frac{\sin(z)}{z}} \leq e^{\abs{\Im(z)}}.\qedhere
\end{equation*}
\end{proof}

\section{Proof of Theorem \ref{t:main-theorem}}
By Lemma \ref{l:explicit-formula} and Lemma \ref{l:difference} we have
\begin{align*}
\pi_M(e^y) - \log(y) - M 
  &= \pi^*_M(e^y) - \log(y) - C_0 + r_M(e^y) \\ 
  &\leq - \zsum_\rho \Eit(-\rho y) + \frac{1+5.4\times 10^{-10}}{y}e^{-y/2}
\end{align*}
for $y > 200$, where we estimated the integral in \eqref{e:explicit-formula} trivially by $e^{-3y}$. Therefore
\begin{equation*}
\int_{\omega-\eps}^{\omega+\eps} K_{c,\eps}(y-\omega) y e^{y/2} \bigl[\pi_M(e^y) - \log(y) - M\bigr]\, dy 
\leq - \sum_\rho \Phi_{\omega,\rho,0} + 1 + 5.4\times 10^{-10},
\end{equation*}
with $\Phi_{\omega,\rho,0}$ as defined in Lemma \ref{l:I-asymp} with $K=K_{c,\eps}$ and $\hat K = \ell_{c,\eps}$. We subdivide the sum over zeros into two parts. For $0<\gamma \leq H$ we choose $k=2$ and $m=0$ in Lemma \ref{l:I-asymp}, which gives
\begin{multline}
-\sum_{\abs{\gamma} \leq H} \Phi_{\omega,\rho,0} \leq
 \sum_{\abs{\gamma} \leq c/\eps} e^{-i\gamma \omega} \ell_{c,\eps}(\gamma)\Bigl(\frac{1}{\rho} - \frac{1}{\omega\rho^2}\Bigr) 
 + \sum_{\frac c\eps < \abs{\gamma} \leq H} \abs{\frac{\ell_{c,\eps}(\gamma)}{\gamma}}\Bigl(1 + \frac{\eps}{c \omega}\Bigr) \\
  + \frac{1}{\omega^2} \sum_{\abs{\gamma}<H}\Bigl( \frac{2.72\eps}{\gamma^2} + \frac{2.01}{\abs{\gamma}^3}\Bigr),
\end{multline}
where we have used $\eps\leq 10^{-3}$. For $\gamma>H$ we have
\begin{multline}
\sum_{\abs{\gamma} >H} \abs{\Phi_{\omega,\rho,0}} \leq e^{h \omega /2} \sum_{\abs{\gamma} > H} \abs{\frac{\ell_{c,\eps}(\frac{\rho}{i} - \frac{1}{2i}) }{\gamma}} \sum_{j=1}^k \frac{(j-1)!}{\omega^{j-1}} H^{1-j} \\
+ e^{h \omega/2}  \sum_{\abs{\gamma} > H} \sum_{j=2}^k \frac{(j-1)!}{\abs{\gamma}^j}\left( \sum_{n=1}^m \binom{n+j-2}{n}  \frac{\abs{\ell^{(n)}_{c,\eps}(\frac{\rho}{i} - \frac{1}{2i}) }}{\omega^{n+j-1}} + \frac{e^{j-2+\eps/2}(e\eps)^{m+1}}{\omega^{j+m-1}(\omega-e\eps)}\right) \\
+ e^{h \omega/2} \sum_{\abs{\gamma}> H}	\frac{k!e^{\eps/2}}{(\omega-\eps)^k \abs{\gamma}^{k+1}}
\end{multline}
for arbitrary $k\geq 2$ and $m\geq 1$, where $h=0$ if the Riemann hypothesis holds and $h=1$ otherwise. So the inequality in \eqref{e:mean-upper-bound2} holds with
\begin{equation}\label{e:E1-def}
\mathcal E_1 = \sum_{\frac c\eps < \abs{\gamma} \leq H} \abs{\frac{\ell_{c,\eps}(\gamma)}{\gamma}}\Bigl(1 + \frac{\eps}{c \omega}\Bigr) + e^{h \omega /2} \sum_{\abs{\gamma} > H} \abs{\frac{\ell_{c,\eps}(\frac{\rho}{i} - \frac{1}{2i}) }{\gamma}} \sum_{j=1}^k \frac{(j-1)!}{\omega^{j-1}} H^{1-j},
\end{equation}
\begin{equation}\label{e:E2-def}
\mathcal E_2 = \frac{1}{\omega^2} \sum_{\rho}\Bigl( \frac{2.72\eps}{\gamma^2} + \frac{2.01}{\abs{\gamma}^3}\Bigr) +  e^{h \omega/2} \sum_{\abs{\gamma}> H}	\frac{k!e^{\eps/2}}{(\omega-\eps)^k \abs{\gamma}^{k+1}},
\end{equation}
and
\begin{equation}\label{e:E3-def}
\mathcal E_3 = e^{\omega/2}  \sum_{\abs{\gamma} > H} \sum_{j=2}^k \frac{(j-1)!}{\abs{\gamma}^j}\left( \sum_{n=1}^m \binom{n+j-2}{n}  \frac{\abs{\ell^{(n)}_{c,\eps}(\frac{\rho}{i} - \frac{1}{2i}) }}{\omega^{n+j-1}} + \frac{e^{j-2+\eps/2}(e\eps)^{m+1}}{\omega^{j+m-1}(\omega-e\eps)}\right).
\end{equation}

We proceed by bounding $\mathcal E_k$. To this end we choose $k=m = \floor{ \omega/2}$. In \eqref{e:E1-def} we take $H=\frac c\eps$, which gives
\begin{equation}\label{e:E1-aux}
\mathcal E_1 \leq  e^{h \omega/2}\sum_{\frac c\eps < \abs{\gamma}} \abs{\frac{\ell_{c,\eps}(\gamma)}{\gamma}} \sum_{j=0}^{k-1} \frac{j!}{\omega^j} \Bigl(\frac \eps c\Bigr)^j,
\end{equation}
where the inner sum is bounded by
\begin{equation*}
\sum_{j=0}^\infty \Bigl(\frac{\eps}{2 c}\Bigr)^j \leq \Bigl(1-\frac{1}{6000}\Bigr)^{-1} \leq 1.0002,
\end{equation*}
since $c\geq 3$. Using this and  \eqref{e:logan-sum1} in \eqref{e:E1-aux} gives \eqref{e:E1-bound}.

In \eqref{e:E2-def} we use the bounds $\sum_\gamma \gamma^{-2} < 0.0463$ and $\sum_{\gamma} \abs{\gamma}^{-3} < 0.00167$ from \cite[Lemma 17]{Rosser41}, the bound
\begin{equation}\label{e:lehman-lemma2}
\sum_{\abs{\gamma}>T} \abs{\gamma}^{-k} \leq T^{1-k} \log(T)
\end{equation}
for $T\geq 2\pi e$ and $k\geq 2$ from \cite[Lemma 2]{Lehman66}, and the inequality $(\omega - \eps)^k \geq e^{-\eps}\omega^k$, which follows from $k\leq \omega/2$,  and get
\begin{equation*}
\mathcal E_2 \leq \frac{0.00336 + 0.126\eps}{\omega^2} + e^{\omega/2} \frac{e^{2 \eps}k!}{(\omega H)^k}\log(H) \leq  \frac{3.36 + 126\eps}{1000\,\omega^2} + 2.8 \Bigl(\frac{e}{2H}\Bigr)^{\omega/2-1} \log(H).
\end{equation*}

In \eqref{e:E3-def} we use \eqref{e:lehman-lemma2} again and the bound from Lemma \ref{l:logan-bound}, where we choose $\delta = 4\eps$, which gives
\begin{multline}
\mathcal E_3 \leq e^{\omega/2} \sum_{j=2}^k H^{1-j}\log(H)\Biggl( \frac{c \, e^{3.12\sqrt{c\eps}}}  {\sinh(c)}\sum_{n=1}^m \frac{j-1}{\omega}\frac{(n+j-2)!}{\omega^{n+j-2}} 2^{-n} \\
+ \frac{1.002 e^{j-1}}{e} \frac{(j-1)!}{\omega^{j-1}} \Bigl(\frac{e\eps}{\omega}\Bigr)^{m+1}\Biggr)
\end{multline}
Since $n+j-2 \leq \omega$ we have $(n+j-2)!/\omega^{n+j-2} \leq 1/\omega$, so the inner sum is bounded by $1/(2\omega)$. In the second summand, we use the bound $(j-1)!/\omega^{j-1}\leq 2^{1-j}$. Since $\sum_{j=1}^\infty H^{-j} \leq 1.001/H$, $\sum_{j=1}^\infty (2 H/e)^{-j} \leq 1.001 e/(2H)$, and $m+1\geq \omega/2$, we obtain the bound in \eqref{e:E3-bound}.

 Finally, the estimate in \eqref{e:ace} follows from \eqref{e:logan-sum2} since
\begin{equation*}
 \sum_{\frac{a c}\eps < \abs{\gamma} \leq \frac c\eps} \abs{\frac{\ell_{c,\eps}(\gamma)}{\rho}\Bigl(1 - \frac{1}{\omega\rho}\Bigl) } \leq \Bigl(1 + \frac{1}{200\times 1000}\Bigr) \sum_{\frac{a c}\eps < \abs{\gamma} \leq \frac c\eps} \abs{\frac{\ell_{c,\eps}(\gamma)}{\gamma}}
\end{equation*}
\qed

\section{Numerical Results}
To locate potential regions where the left hand side of \eqref{e:mean-upper-bound2} should be small, the function
\begin{equation*}
\sigma_T(y) = \sum_{\abs{\gamma}\leq T}  \frac{e^{i\gamma y}}{\frac 12 - i\gamma}.
\end{equation*}
has been evaluated for $T=10^{6}$ at all points in $10^{-7}\Z \cap [1,2500]$. Since $\ell_{c,\eps}(\gamma) = 1 + O((\eps\gamma)^2/c)$ for $\gamma\to 0$ this gives a reasonably good approximation to the first part of the sum in \eqref{e:mean-upper-bound2}, and the objective is thus to find regions where $\sigma_T(y)$ smaller than $-1$. 

The evaluation has been done using the method for fast multiple evaluation of trigonometric sums from \cite{FKBJ}. A more detailed search with $T=10^8$ around $495.7028078$, the first point where $\sigma_{10^6}(y)$ turned out to be promisingly small, revealed a short region of length $\approx 2.8\times 10^{-8}$ about $495.702833137$ where $\sigma_{10^{8}}(y)$ is smaller than $-1$.

Theorem \ref{t:explicit-upper-bound} now follows by an application of Theorem \ref{t:main-theorem} with $\omega = 495.702833137$, $c=280$, $\eps = 2.8\times 10^{-8}$, $H=10^{11}$ (which has been reported in \cite{FKBJ}) and $a= 0.4$.

The sum over zeros was calculated using approximations to the zeros with imaginary part up to $4\times 10^{9}$ which were given within an absolute accuracy of $2^{-64}$. The sum was evaluated using multiple precision arithmetic, which gave the bound
\begin{equation}
 \sum_{\abs{\gamma} \leq 4\times 10^{9}} e^{-i\gamma \omega} \ell_{c,\eps}(\gamma)\Bigl(\frac{1}{\rho} - \frac{1}{\omega\rho^2}\Bigr) \leq -1.00015419.
\end{equation}
The sum in \eqref{e:ace} is then bounded by $1.2\times 10^{-11}$ and we have
\[
\mathcal  E_1 + \mathcal E_2 +\mathcal  E_3 \leq 1.2\times 10^{-12} + 1.37\times 10^{-8} + 1.6\times 10^{-24} \leq 1.38\times 10^{-8}.
\]
Thus, the left hand side of \eqref{e:mean-upper-bound2} is bounded by
\[
 -1.00015419 + 1.2\times 10^{-11} + 1 + 5.4\times 10^{-10} + 1.38\times 10^{-8} < -0.000154.
\]
Consequently, there exists an $x\in [\exp(w-\eps), \exp(w+\eps)]$ such that $\pi_M(x) - \log\log(x) - M < -0.000154 /(\sqrt{x}\log x)$. Obviously, we have
\begin{align*}
 \pi_M(x-y) - \log\log(x-y) - M &\leq \pi_M(x) - \log\log(x) - M + \int_{x-y}^x \frac{dt}{t\log t} \\
 &\leq  - \frac{0.000154}{\sqrt{x}\log(x)} + \frac{y}{(x-y)\log (x-y)},
\end{align*}
which is negative for $y \leq 0.00015\sqrt{x}$. Since $0.00015\sqrt{x} > \exp(239.046541)$ the assertion of Theorem \ref{t:explicit-upper-bound} follows.
\qed

\renewcommand{\arraystretch}{1.2}
\begin{table}
\caption{Values of $y\in [1,2500]$ for which $\sigma_{10^6}(y) < -0.95$.}
\begin{tabular}{r|r}
$y$	& $\sigma_{10^6}(y)$ \\
\hline
$495.7028078$		&	$-0.9972\dots$ \\
$1423.957207$		&	$-0.9740\dots$ \\
$1623.9204309$		&	$-0.9807\dots$ \\
$1859.1291846$		&	$-1.0511\dots$ \\
$2107.5263606$		&	$-1.0214\dots$ \\
$2285.3917834$		&	$-1.0454\dots$ \\
$2430.3039554$		&	$-1.0172\dots$ \\
$2447.6661764$		&	$-1.0028\dots$ \\
\end{tabular}

\end{table}

\section*{Acknowledgment}
The author wishes to thank the anonymous referee for his or her quick and thorough proofreading.

\bibliographystyle{amsalpha}
\bibliography{../../TeX/inputs/jankabib}
\end{document}